\documentclass[utf8,11pt]{article}
\usepackage[utf8]{inputenc}
\usepackage{fullpage}
\usepackage{authblk}

\usepackage{amsthm,amsmath,amssymb}

\newcommand{\qedclaim}{\hfill $\diamond$ \medskip}
\newenvironment{proofclaim}{\noindent{\em Proof of the claim.}}{\qedclaim}

\usepackage[colorlinks=true,citecolor=black,linkcolor=black,urlcolor=blue]{hyperref}

\theoremstyle{plain}
\newtheorem{theorem}{Theorem}

\newtheorem{claim}{Claim}[theorem]

\theoremstyle{definition}

\newtheorem{conjecture}[theorem]{Conjecture}

\begin{document}

\title{Immersion of transitive tournaments in digraphs with large minimum
outdegree\footnote{This work was supported by ANR under contract STINT ANR-13-BS02-0007.}}

\author{ William Lochet$^{a,b}$\\~\\
\small $^a$ Universit\'e C\^ote d'Azur, CNRS, I3S, UMR 7271, \\ \small Sophia Antipolis, France\\~\\
\small $^b$Laboratoire d'Informatique du Parall\'elisme \\ \small   UMR 5668 ENS Lyon - CNRS - UCBL - INRIA
\\ \small Universit\'e de Lyon, France\\~\\
}

\date{}

\maketitle
\begin{abstract}
We prove the existence of a function $h(k)$ such that every simple digraph with 
minimum outdegree greater than $h(k)$ contains an immersion of the transitive 
tournament on $k$ vertices. This solves a conjecture of Devos, McDonald, Mohar and Scheide.
\end{abstract}

In this note, all digraphs are without loops. A digraph $D$ is \textit{simple} if 
there is at most one arc from $x$ to $y$ for any $x,y \in V(D)$. Note that arcs in opposite 
directions are allowed. The \textit{multiplicity} of a digraph $D$ is the maximum number of parallel arcs 
in the same direction in $D$.  
We say that a digraph $D$ contains an \textit{immersion} of a digraph $H$ if the vertices of $H$ are
mapped to distinct vertices of $D$, and the arcs of $H$ are mapped to directed
paths joining the corresponding pairs of vertices of $D$, in such a way that
these paths are pairwise arc-disjoint. If the directed paths are vertex-disjoint, we say that 
$D$ contains a \textit{subdivision} of $H$. 

Understanding the necessary conditions for graphs to contain a subdivision of a clique is a very natural 
and well-studied question. One of the most important examples is the following result by Mader \cite{Mader67}:

\begin{theorem}[\cite{Mader67}]
For every $k \geq 1$, there exists an integer $f(k)$ such that every graph with minimum
degree greater than $f(k)$ contains a subdivision of $K_k$.
\end{theorem}
Bollob\'{a}s and Thomason \cite{Bollo} as well as Koml\'{o}s and Szemer\'edi \cite{Kom} proved that  $f(k) = O(k^2)$.
In the case of digraphs, there exist examples of digraphs with large out- and indegree
without a subdivision of the complete digraph on three vertices, as shown by Thomassen \cite{Thomassen}.
However Mader \cite{Mader85}
conjectured that an analogue should hold for transitive
tournaments $TT_{k}$ in digraphs with large minimum outdegree.

\begin{conjecture}[\cite{Mader85}]
For every $k\geq 1$, there exists an integer $g(k)$ such that every simple digraph with minimum
outdegree at least $g(k)$ contains a subdivision of $TT_k$.
\end{conjecture}

The question turned out to be way more difficult than the non oriented case, as
the existence of $g(5)$ remains unknown. Weakening the statement, Devos, McDonald,
Mohar and Scheide \cite{Devos}
made the following conjecture replacing subdivision with immersion and proved
it for the case of eulerian digraphs. 

\begin{conjecture}[\cite{Devos}]
For every $k \geq 1$, there exists an integer $h(k)$ such that every simple digraph with minimum
outdegree at least $h(k)$ contains an immersion of $TT_k$.
\end{conjecture}

Finding the right value for $h(k)$ in the case of non oriented graphs is an interesting question on its own
(see \cite{Devosnono} for more details).

The goal of this note is to present a proof of this conjecture.
Let $F(k,l)$ be the digraph consisting of $k$ vertices $x_1, \dots, x_k$ such that
there exists $l$ arcs from $x_i$ to $x_{i+1}$ for every $1 \leq i \leq k-1$.
It is clear that $F(k,\binom{k}{2})$ contains an immersion of $TT_k$, so the following theorem
implies Conjecture 3.

\begin{theorem}
For every $k \geq 1$ and $l$, there exists a function $f(k,l)$ such that every digraph with minimum outdegree greater
than $f(k,l)$ and multiplicity at most $kl$ contains an immersion of $F(k,l)$.
\end{theorem}

\begin{proof}
We prove the result for $f(k,l) = 2k^3l^2$ and $l \geq 2$.
We proceed by induction on $k$.
For $k=1$ this is trivial because $F(1,l)$ is one vertex. Suppose now that the result
holds for $k$ and assume for a contradiction that it does not hold for $k+1$. 
Let $D$ be the digraph with the smallest number of arcs and vertices such that $D$ has multiplicity at most $(k+1)l$,
all but at most $c_1 = k + (k+1)l$ vertices have outdegree at least $f(k+1,l)$ and without an immersion of $F(k+1,l)$. 
By minimality of $D$, every vertex has outdegree exactly $f(k+1,l)$,
expect $c_1$ of them with outdegree 0. Call $T$ the set of vertices of outdegree 0.
By removing $T$ and some of the parallel arcs, 
we obtain a digraph of outdegree greater than $ d' = f(k+1,l) - c_1(k+1)l - \frac{f(k+1,l)}{k+1} $ with multiplicity $kl$.
Because $f(k+1,l) - f(k,l) =  2(3k^2 + 3k + 1)l^2$ and $c_1(k+1)l +\frac{f(k+1,l)}{(k+1)} = k(k+1)l + 3(k+1)^2l^2$, 
we get that  $d' \geq f(k,l)$ and by induction there exists an immersion of $F(k,l)$ in $D - T$.
Call $X = \{x_1, \cdots, x_k\}$ the set of vertices of the immersion and $P_{i,j}$ the $jth$ directed path of
this immersion from $x_i$ to $x_{i+1}$. We can assume this immersion is of minimum size, so that 
every vertex in $P_{i,j}$ has exactly one outgoing arc in $P_{i,j}$. Let $D'$ be the digraph
obtained from $D$ by removing all the arcs of the $P_{i,j}$ and the vertices $x_1, \dots, x_{k-1}$. 
By the previous remark, the degree of each vertex in $D'$ is either $0$ if this vertex belongs to $T$ or at least $ f(k+1,l) - (k-1)l -(k-1)(k+1)l$.

For every vertex $y \in D' - x_k$, there do not exist $l$ arc-disjoint
directed paths from $x_k$ to $y$ in $D'$, for otherwise there would be an immersion
of $F(k+1,l)$ in $D$.
Hence, by Menger's Theorem there exists a set $E_y$ of less than 
$l$ arcs such that there is no directed path from $x_k$ to $y$ in $D' \setminus E_y$.
Define $C_y$ for every vertex $y \in D' - x_k$
as the set of vertices which can reach $y$ in $D' \setminus E_y$.
Now take $Y$ a minimal set such that  $\cup_{y \in Y}C_y$ covers $D' - x_k$. We claim
that $Y$ consists of at least $c_2 \geq \frac{f(k+1,l) -(k-1)l - (k-1)(k+1)l}{l} \geq 2c_1$ elements, as
$ \cup_{y \in Y} E_y $ must contain all the arcs of $D'$ with $x_k$ as tail.

For each $y \in Y$, define $S_y$ as the set of vertices which belong to $C_y$ and no other $C_{y'}$ for $y' \in Y$.
Since $Y$ is minimal, every $S_y$ is non-empty.
Note that for $u \in S_y$, if there exists $y' \in Y \setminus y$ and $v \in C_{y'}$ such that $uv \in A(D)$, then $uv \in E_{y'}$. 
Note that $T \subset Y$ as vertices in $T$ have outdegree 0 and if
$y \in Y \setminus T$ then $S_y$ consists only of vertices of outdegree $f(k+1,l)$ in $D$.

Let $R$ be the digraph with vertex set $Y$ and arcs from $y$ to $y'$ if there is an arc from $S_y$ to $C_{y'}$. 
As noted before, $d^-_R(y) \leq | E_y | \leq l$. The average outdegree of the vertices
of $Y\setminus T$ in $R$ is then at most $\frac{c_1l + (c_2 - c_1) l  }{c_2-c_1} \leq 2l$. Let
$y$ be a vertex of $R \setminus T $ with outdegree at most this average.
Let $H$ be the digraph induced on $D'$ by the vertices in $Sy$ to which we add $X$, all the arcs 
that existed in $D$ (with multiplicity) from vertices of $S_y$ to vertices of $X$ and 
the following arcs: For each $P_{i,j}$, let $z_1, z_2,
\dots, z_l = P_{i,j} \cap Sy $, where $z_i$ appears before $z_{i+1}$ on $P_{i,j}$
and add all the arcs $(z_i, z_{i+1})$ to $H$. Note that, if $(x,y)$ is an arc of $D'$,
then by minimality of the copy of $F(k,l)$, every time $x$ appears before $y$ on
some $P_{i,j}$, then $P_{i,j}$ uses one of the arcs $(x,y)$. Thus for each pair
of vertices $x$ and $y$ in $H$, either $(x,y) \in A(D)$ and the number of $(x,y)$ arcs
in $H$ is equal to the one in $D$, or $(x,y) \not \in A(D)$ and the number of $(x,y)$ arcs
in $H$ is bounded by $(k-1)l$. This implies that $H$ has multiplicity at most $(k+1)l$.

\begin{claim}
$H$ is a digraph with multiplicity at most $(k+1)l$, such that all but at most $c_1$ vertices have outdegree greater
than $f(k+1,l)$ and $H$ does not contain an immersion of $F(k+1,l)$. 

\end{claim}

\begin{proofclaim}
Suppose $H$ contains an immersion of $F(k+1,l)$, then by replacing the new arcs
by the corresponding directed paths along the $P_{i,j}$ we get an immersion of $F(k+1,l)$ in $D$.
Moreover, we claim that the number of vertices in $H$ with outdegree smaller than $f(k+1, l)$
is at most $k + 2l+ (k-1)l = c_1$. Indeed, the vertices
of $H$ that can have outdegree smaller in $H$ than in $D$ are the $x_i$, or the vertices with
outgoing arcs in $E_{y'}$ for some $y' \in Y\setminus y$, or the vertices
along the $P_{i,j}$. But with the additions of the new arcs, we know that there
is at most one vertex per path $P_{i,j}$ that loses some outdegree in $H$. 
\end{proofclaim}

However, since $H$ is strictly smaller than $D$, we reach a contradiction.

\end{proof}

\section*{Acknowledgements}

The author wishes to thank Fr\'ed\'eric Havet and St\'ephan Thomass\'e for their useful comments on the manuscript.

\end{document}